\documentclass[11pt]{amsart}

\usepackage[a4paper,hmargin=3.5cm,vmargin=4cm]{geometry}
\usepackage{amsfonts,amssymb,amscd,amstext}
\usepackage{graphicx}
\usepackage[dvips]{epsfig}
\usepackage{quoting}
\usepackage{hyperref}

\usepackage{fancyhdr}
\input xy
\xyoption{all}
\usepackage{tikz-cd}

\usepackage[color=blue!20]{todonotes}

\usepackage{enumitem}
\usepackage{titlesec}
\usepackage{mathrsfs}

\titleformat{\section}
{\filcenter\bfseries} {\thesection{.}}{0.2cm}{}
\titleformat{\subsection}[runin]
{\bfseries} {\thesubsection{.}}{0.15cm}{}[.]
\titleformat{\subsubsection}[runin]
{\em}{\thesubsubsection{.}}{0.15cm}{}[.]

\usepackage[up,bf]{caption}
\usepackage{subcaption}

\newtheorem{theorem}{Theorem}[section]

\newtheorem{lemma}[theorem]{Lemma}

\theoremstyle{definition}
\newtheorem{definition}[theorem]{Definition}
\newtheorem{remark}[theorem]{Remark}

\newtheorem{problem}[theorem]{Problem}

\numberwithin{equation}{section}
\numberwithin{figure}{section}

\newcommand\Acal{\mathcal{A}}

\newcommand\Mcal{\mathcal{M}}

\newcommand\Tcal{\mathcal{T}}

\newcommand\Ascr{\mathscr{A}}

\newcommand\Cscr{\mathscr{C}}

\newcommand\C{\mathbb{C}}
\newcommand\D{\overline{\mathbb D}}

\renewcommand\D{\mathbb D}

\renewcommand\L{\mathbb{L}}

\newcommand\N{\mathbb{N}}

\newcommand\R{\mathbb{R}}

\renewcommand\c{\mathbb{C}}

\newcommand\h{\mathbb{H}}
\newcommand\n{\mathbb{N}}
\renewcommand\r{\mathbb{R}}
\newcommand\s{\mathbb{S}}

\newcommand\igot{\mathfrak{i}}

\renewcommand\igot{\mathfrak{i}}
\newcommand\mgot{\mathfrak{m}}

\newcommand\Agot{\mathbf{A}}

\newcommand\wt{\widetilde}
\newcommand\wh{\widehat}

\newcommand\dist{\mathrm{dist}}

\newcommand\length{\mathrm{length}}

\newcommand\CMC{\mathrm{CMC\text{-}1}}
\newcommand{\SL}{\text{SL}_2(\C)}

\def\dist{\mathrm{dist}}

\def\length{\mathrm{length}}

\usepackage{color}

\usepackage{array}

\usepackage[inkscapeformat=png]{svg}

\definecolor{dark}{rgb}{0,0.55,0}
\definecolor{darkblue}{rgb}{0,0,0.55}
\definecolor{orange}{rgb}{0.8,0.5,0}
\definecolor{darkred}{rgb}{0.7,0,0}

\begin{document}
	
	
	\fancyhead[LO]{$\CMC$ surfaces in hyperbolic and de Sitter spaces with Cantor ends
	}
	\fancyhead[RE]{I.\ Castro-Infantes and J.\ Hidalgo}
	\fancyhead[RO,LE]{\thepage}
	
	\thispagestyle{empty}
	
	
	
	\begin{center}
		{\bf\Large $\CMC$ surfaces in hyperbolic and  \\
			de Sitter spaces with Cantor ends 
		} 
		
		\bigskip
		
		%
		%
		{\bf
			Ildefonso Castro-Infantes and
			Jorge Hidalgo}
	\end{center}
	
	
	%
	%
	
	\bigskip
	
	\begin{quoting}[leftmargin={7mm}]
		{\small
			\noindent {\bf Abstract}\hspace*{0.1cm}
			We prove that on every compact Riemann surface $M$ 
			there is a Cantor set $C \subset M$ such that $M \setminus C$ admits a proper conformal constant mean curvature one ($\CMC$) immersion into hyperbolic $3$-space $\h^3$. 
			Moreover, we obtain that every bordered Riemann surface admits an almost proper $\CMC$ face into de Sitter $3$-space $\s_1^3$, and we show that on every compact Riemann surface $M$ 
			there is a Cantor set $C \subset M$ such that $M \setminus C$ admits  an almost proper $\CMC$ face into $\s_1^3$.
			These results follow from different uniform approximation theorems for holomorphic null curves in $\C^2 \times \C^*$ that we also establish in this paper. 
			
			\noindent{\bf Keywords}\hspace*{0.1cm} 
			 $\CMC$ surface, CMC-1 face, Riemann surface, Cantor set, holomorphic null curve
			
			
			\noindent{\bf Mathematics Subject Classification (2020)}\hspace*{0.1cm} 
			53A10, 
			53C42, 
			32H02, 
			32E30. 
		}
	\end{quoting}
	
	
	
	\section{Introduction}
	In 1987 Bryant \cite{Bryant1987Asterisque} introduced a holomorphic representation for constant mean curvature one ($\CMC$ from now on) surfaces in hyperbolic $3$-space $\h^3$.
	 Using this representation, Alarcón and López proved in 2013 that every open orientable surface admits a complete bounded $\CMC$ immersion into $\h^3$, see \cite[Corollary III]{AlarconLopez2013MA}. Furthermore, Alarcón and Forstneri\v c showed in 2015 that every bordered Riemann surface (see Definition \ref{def:bordered}) admits a proper conformal $\CMC$ immersion into $\h^3$, see \cite[Corollary 3]{AlarconForstneric2015MA}. These results led Alarcón and Forstneri\v c to pose the next question. 
	\begin{problem}{\cite[Problem 1]{AlarconForstneric2015MA}}\label{p1}
		{\em Does every open Riemann surface admit a proper conformal $\CMC$ immersion into $\h^3$?}
	\end{problem}
	Our first main result contributes to Problem \ref{p1} by providing the first known examples of properly immersed $\CMC$ surfaces in $\h^3$ with Cantor ends.  
	Recall that a {\em Cantor set} is a compact, totally disconnected set (every connected component is just a point) with no isolated points. By {\em compact surface} we mean that the surface is topologically compact and has empty boundary.
	\begin{theorem}\label{th:mainH}
		Let $M$ be a compact Riemann surface. There exists a Cantor set $C \subset M$ such that $M \setminus C$ admits	
		a proper conformal $\CMC$ immersion into $\h^3$.
	\end{theorem}
	Theorem \ref{th:mainH} also holds when $M$ is a bordered Riemann surface, see Remark \ref{rmk:bd} and the proof of Theorem \ref{th:mainH} in Section \ref{sec:HS}. It is obtained as an application of Theorem \ref{th:cantor}, a Runge approximation type result for {\em holomorphic null curves}.  A holomorphic null curve in $\C^3$ is a holomorphic immersion $M \to \C^3$ from an open Riemann surface $M$ into the complex Euclidean $3$-space whose derivative with respect to any local holomorphic coordinate on $M$ takes values in $\Agot_* =\Agot\setminus \{0\}$, where $\Agot \subset \C^3$ is the {\em null quadric} 
	\begin{equation}\label{def:nullquadric}
		\Agot = \{ (z_1, z_2, z_3) \in \C^3 \colon z_1^2 + z_2^2 + z_3^2  = 0\}.
	\end{equation}
	Holomorphic null curves are strongly related to {\em minimal surfaces} in $\R^3$, this is, locally area minimizing surfaces, or equivalently, constant mean curvature zero surfaces.
	More precisely, the real and imaginary parts of a holomorphic null curve are conformal minimal immersions into $\r^3$, and every such immersion is locally the real part of a holomorphic null curve.
	This allows to use complex analytic methods for studying minimal surfaces; we refer to \cite{AlarconForstneric2019JAMS,AlarconForstnericLopez2021Book} for details.

	Similarly, there is a projection $ \C^2 \times \C^* \to \h^3$ ($\C^* = \C \setminus \{ 0\}$) arising from Bryant's representation and which takes holomorphic null curves into conformal $\CMC$ immersions in $\h^3$. Conversely, every $\CMC$ surface in $\h^3$ locally lifts to a holomorphic null curve in $\C^2 \times \C^*$. This projection, which is better described in Section \ref{sec:HS}, considerably simplifies the task of constructing $\CMC$ surfaces with control on the complex structure via complex analytic methods, see \cite{AlarconCastro-InfantesHidalgo2023} and the references therein.
	For instance, the fact that every bordered Riemann surface admits a proper conformal  $\CMC$ immersion in $\h^3$ follows from this projection and the following result proved by Alarcón and Forstneri\v c  in 2015: given constants $0 <c_1<c_2$ and a bordered Riemann surface $M$, there is a proper holomorphic null curve $X = (X_1, X_2, X_3) \colon M \to \C^3$ with $c_1 < |X_3|< c_2$ on $M$, see \cite[Theorem 2]{AlarconForstneric2015MA}.

Not every open Riemann surface admits a holomorphic null curve in $\C^3$ with bounded third component, hence  a different result in $\C^2 \times \C^*$ is needed to approach Problem \ref{p1} via the projection $\C^2 \times \C^* \to \h^3$.
In \cite{AlarconCastro-InfantesHidalgo2023},  Alarcón and the authors proved a Runge approximation result for holomorphic null curves in $\C^3$, following the approach in \cite{AlarconLopez2012JDG,AlarconForstneric2014IM,AlarconCastro-Infantes2019APDE}, but with an additional idea to control the zero set of the third component of the holomorphic null curves in $\C^3$. As a consequence, it was proved that every open Riemann surface admits a complete conformal $\CMC$ immersion into $\h^3$, see \cite[Theorem 1.1]{AlarconCastro-InfantesHidalgo2023}.
However, properness could not be achieved by technical reasons. 
Now, applying ideas of \cite{AlarconForstneric2015MA} and \cite{Forstneric2022RMI} we show that these technical difficulties can be overcome to construct proper conformal $\CMC$ surfaces in $\h^3$ with Cantor ends.

In this paper we also consider $\CMC$ spacelike surfaces in de Sitter $3$-space $\s_1^3$ with certain singularities.  Aiyama and Akutagawa gave in \cite{AiyamaAkutagawa1999AGAG} a holomorphic representation for conformal $\CMC$ spacelike immersions $M \to \s_1^3$ of Riemann surfaces into de Sitter space.  However, the only complete such surface  is the so called $\s_1^3$-horosphere, see \cite{Akutagawa1987, FRUYY09}, which is totally umbilical.
So, in order to have a rich global theory for spacelike $\CMC$ surfaces in $\s_1^3$ one needs to consider a wider class of surfaces than just complete and immersed ones. In 2006, Fujimori introduced the notion of {\em $\CMC$ face}, see  \cite[Definition 1.4]{Fujimori06}, later characterized in 2013, see \cite[Definition 1.1]{FKKRUY13}. According to the latter, a $\CMC$ face is a smooth map $M \to \s_1^3$, where $M$ is an open Riemann surface, which is a conformal spacelike immersion on an open dense subset of $M$ and whose differential does not vanish at any point. The points where $M \to \s_1^3$ is not an immersion are called {\em singular points}. 
Fujimori gave a holomorphic representation for $\CMC$ faces in \cite[Theorem 1.9]{Fujimori06}, which provides a projection $ \C^2 \times \C^* \to \s_1^3$ carrying holomorphic null curves into $\CMC$ faces, see Section \ref{sec:HS} for details. In view of this projection, and the strategy already exploited to construct $\CMC$ surfaces in $\h^3$, the analogue to Problem \ref{p1} can be reasonably posed.
\begin{problem}\label{p2}
		{\em Does every open Riemann surface admit a proper $\CMC$ face into $\s_1^3$?}
\end{problem}
There are many examples of $\CMC$ faces, see \cite{LeeYang06OJM, Fujimori06, MR2317280, FRUYY09RM, FKKRUY13,  MR3614568, MR3965794, MR4451410} for instance. However,  only few Riemann surfaces were known to be the complex structure of a $\CMC$ face until recently, when, using the projection $ \C^2 \times \C^* \to \s_1^3$,  Alarcón and the authors showed that every open Riemann surface $M$ admits a {\em weakly complete} \cite[Def. 1.3]{FRUYY09} $\CMC$ face $M \to \s_1^3$, see \cite[Corollary 5.3]{AlarconCastro-InfantesHidalgo2023}. In this paper we go further in this direction and give a first approach to Problem \ref{p2}. 
\begin{theorem}\label{th:bds}
	Every bordered Riemann surface $M$ admits a weakly complete almost proper $\CMC$ face into $\s^3_1$.
\end{theorem}
Recall that a continuous map $f:X\to Y$ between topological spaces is {\em almost proper} if for every compact set $K\subset Y$, the connected components of $f^{-1}(K)$ are all compact. 
Theorem \ref{th:bds} is the analogue to \cite[Corollary 3]{AlarconForstneric2015MA} for weakly complete almost proper $\CMC$ faces, and it is also proved in Section \ref{sec:HS} as a consequence of Theorem \ref{th:bordered}, another Runge approximation theorem for holomorphic null curves. For the latter, we first prove Lemma \ref{lemma:apds}, where we follow the ideas in \cite[Lemma 4]{AlarconForstneric2015MA} but with special care when applying the Riemann-Hilbert theorem for holomorphic null curves, see \cite[Theorem 6.4.2]{AlarconForstnericLopez2021Book}.

The almost proper $\CMC$ faces in Theorem \ref{th:bds} are weakly complete
by construction, because they are projections of complete holomorphic null curves in $\SL$, see  Section \ref{sec:HS} and \cite{Yu97}.
Since not every weakly complete $\CMC$ face is almost proper, Theorem \ref{th:bds} improves \cite[Corollary 5.3]{AlarconCastro-InfantesHidalgo2023} when the open Riemann surface is bordered.
We remark that the almost proper $\CMC$ faces in Theorem \ref{th:bds} are not {\em complete} in the sense of  \cite[Def. 1.2]{FRUYY09}, as this would imply that its ends are all conformally equivalent to a punctured disk, see \cite{MR2765562, MR2967008, MR3229091} and \cite[Fact 1.1]{FKKRUY13}.

Finally, as  a consequence of the Runge approximation Theorem \ref{th:cantor2}, we obtain the last main theorem of this paper, which gives the first known examples of almost proper $\CMC$ faces into $\s^3_1$ with Cantor ends.
\begin{theorem}\label{th:mainS}
	Let $M$ be a compact Riemann surface. There exists a Cantor set $C \subset M$ such that $M \setminus C$ admits	
	a weakly complete almost proper $\CMC$ face into $\s^3_1$.
\end{theorem}
Theorem \ref{th:mainS} also holds when $M$ is a bordered Riemann surface, see Remark \ref{rmk:bd} and the proof of Theorem \ref{th:mainS} in Section \ref{sec:HS}.
The question of whether  Theorem \ref{th:bds} and Theorem \ref{th:mainS} hold for proper $\CMC$ faces into $\s^3_1$ remains open.


\section{Preliminaries} \label{sec:prel}

\subsection{Notation and definitions}
We shall denote $\igot = \sqrt{-1}\in\C$, $\R^+ = [0 , + \infty)$, $\n = \{1,2,3,4,\hdots\}$, 
and by $| \cdot |$, $\dist ( \cdot, \cdot)$, and $\length(\cdot)$ the Euclidean norm, distance, and length on $\C^n$ $(n \in \n)$, respectively. We recall the following.

\begin{definition}{(\cite[Def 1.10.8]{AlarconForstnericLopez2021Book})}\label{def:bordered}
A {\em bordered Riemann surface} is an open Riemann surface which is the interior of a compact one dimensional complex manifold $\overline{M}$ with smooth boundary $b \overline{M}  \not = \varnothing$ consisting of finitely many closed Jordan curves. 
\end{definition}
A domain $D$ in an open Riemann surface $M$ is said to be a {\em bordered domain} if it is relatively compact and has smooth boundary; such $D$ is itself a bordered Riemann surface with the complex structure induced from $M$. It is classical that every bordered Riemann surface is biholomorphic to a relatively compact bordered domain in a larger Riemann surface.

Given a set $K$ in a complex manifold we shall say that a map $K \to \C^n$ ($n\in\n)$ is holomorphic if it extends holomorphically to an open neighborhood of $K$. In the same way, we will say that a map $K \to \C^n$ is a holomorphic null curve if it extends to a holomorphic null curve on an open neighborhood of $K$. If $K$ is compact we denote by $\Ascr^r(K)$ ($r\in\n \cup \{0\}$) the space of all $\Cscr^r$ functions $K\to\C$ which are holomorphic in the interior $\mathring K$ of $K$, and write $\Ascr(K)$ for $\Ascr^0 (K)$. Likewise, we define the space $\Ascr^r(K,N)$  for maps into a complex manifold $N$; nevertheless we shall omit the target from the notation when it is clear from the context.
We will also say that a map $K\to\C^n$ is a null curve of class $\Ascr^r (K)$ if it is a  $\Cscr^r$ map which is a holomorphic null curve in $\mathring K$.
For a $\Cscr^r$ map $f\colon K\to \C^n$, we denote by $\| f \|_{r,K}$ the $\Cscr^r$ maximum norm of $f$ on $K$.

The following definitions, see \cite[Def.\ 1.12.9 and 3.1.3]{AlarconForstnericLopez2021Book}, will play a central role in the construction of holomorphic null curves in Section \ref{sec:proper}.

\begin{definition}\label{def:admissible}
	An \emph{admissible set} in a smooth surface $M$ is a compact set of the form $S = K \cup \Gamma \subset M $, where $K$ is a (possibly empty) finite union of pairwise disjoint compact domains with piecewise $\Cscr^1$ boundaries in $M$, and $\Gamma= S \setminus \mathring K =\overline{S\setminus K}$ is a (possibly empty) finite union of pairwise disjoint smooth Jordan arcs and closed Jordan curves meeting $K$ only at their endpoints (if at all) and such that their intersections with the boundary $bK$ of $K$ are transverse.
\end{definition}
\begin{definition}\label{def:gnc}
	Let $S = K \cup \Gamma$ be an admissible set in a Riemann surface $M$, and $\theta$  be a nowhere vanishing holomorphic $1$-form on a neighborhood of $S$. A \emph{generalized null curve} $S \to \C^n$ $(n \geq 3)$ of class $\Ascr^r (S)$ $(r \in \n)$ is a pair $(X,f \theta$) with $X \in \Ascr^r ( S, \C^n)$ and $f \in \Ascr^{r-1} ( S, \mathbf{A}_*)$ (see \eqref{def:nullquadric}) such that $f \theta = dX $ holds on $K$ (hence $X \colon \mathring K \to \C^n$ is a holomorphic null curve), and for a smooth path $\alpha$ in $M$ parameterizing a connected component of $\Gamma$ we
	have $ \alpha^* (f \theta) = \alpha^* (d X)= d ( X \circ \alpha)$.
\end{definition}
Finally, recall that a compact subset $K$ of an open Riemann surface $M$ is said to be \emph{Runge} if every function in $\Ascr (K)$ may be approximated uniformly on $K$ by holomorphic functions defined on $M$. By the Runge-Mergelyan theorem  \cite[Theorem 5]{FFW18} this is equivalent to $M\setminus K$ having no relatively compact connected components in $M$ (i.e., $K$ has no holes). 
A map $f\colon M \to \C^n$ is {\em flat} if $f(M)$ is contained in an affine complex line in $\C^n$, and {\em nonflat} otherwise.

	\subsection{Technical lemmas}
	We shall comprise in this subsection the main technical tools for the proofs in Section \ref{sec:proper}.
	Let $M$ be an open Riemann surface, $K \subset M$ be a subset, and $f =(f_1,f_2,f_3)\colon K \to \C^3$ be a map.  In \cite{AlarconForstneric2015MA}, Alarcón and Forstneri\v c introduced the function $	\mgot (f) \colon K \to \R^+$ given by
	\begin{equation}\label{eq:m}
		\mgot (f) = \frac{1}{\sqrt{2}}\max \{ | f_1 + \igot f_2 | ,| f_1 - \igot f_2|\}.
	\end{equation}
	Note that $\mathfrak{m}(f) \leq |(f_1,f_2)| \le | f |$ on $K$ by Cauchy-Schwarz inequality. 
	The following is \cite[Lemma 3.2]{AlarconCastro-InfantesHidalgo2023}. 
	\begin{lemma}\label{lemma:ap}
		Let $M= \overline M\setminus b\overline M$ be a bordered Riemann surface,  $K\subset M$ be a smoothly bounded compact domain,  $\Lambda \subset M$ be a finite set,  $s_0 >0$ be a number, and  $X = (X_1, X_2, X_3) \colon \overline M \to \C^3$  be a null curve of class $\Ascr^1(\overline M)$. Assume that
			$\mathfrak{m}(X) > s_0$ on $\overline M \setminus \mathring K$.
		Given numbers $\varepsilon>0$, $ s > s_0$, and $k \in \n$, there exists a null curve $\wh X  = ( \wh X_1, \wh X_2, \wh X_3) \colon \overline M \to \C^3$ of class  $\Ascr^1(\overline M)$ satisfying the following:
		\begin{enumerate}[label=\rm (\roman*)]
			\item \label{L1} $\| \wh X - X \|_{1, K} < \varepsilon$,
			\item \label{L2} $\mathfrak{m} (\wh X) > s $ on $b\overline M$,
			\item  \label{L3} $\mathfrak{m}(\wh X) > s_0$ on $\overline M \setminus \mathring K$,
			\item \label{L4} $ \| \wh X_3 - X_3 \|_{0, \overline M } < \varepsilon$,
			\item  \label{L5} $\wh X - X $ vanishes to order $k$ everywhere on $\Lambda$.
		\end{enumerate}
	\end{lemma}
This lemma, together with the projection $\C^2 \times \C^* \to \h^3$, was used to construct proper conformal $\CMC$ immersions  $M\to\h^3$ when $M$ is a bordered Riemann surface, see \cite[Corollary 3]{AlarconForstneric2015MA}; and also to construct almost proper conformal $\CMC$ immersions $M\to \h^3$ when $M$ is an open Riemann surface, see \cite[Corollary 5.2]{AlarconCastro-InfantesHidalgo2023}.
The following lemma is a modification of Lemma \ref{lemma:ap}, replacing conclusion \ref{L3} by a suitable one to construct almost proper $\CMC$ faces in $\s_1^3$.

	\begin{lemma}\label{lemma:apds} 
		Let $M= \overline M\setminus b\overline M$ be a bordered Riemann surface,  $K\subset M$ be a smoothly bounded compact subset,  $\Lambda \subset M$ be a finite set,  $s > s_0 >0$ be numbers, and  $X = (X_1, X_2, X_3) \colon \overline M \to \C^3$  be a null curve of class  $\Ascr^1(\overline M)$ such that
		\begin{equation}\label{eq:boundary}
			\left|  X_1 + \igot X_2   \right| > s_0   \text{ on } b\overline{M}.
		\end{equation}
		Given numbers $\varepsilon>0$ 
		and $k \in \n$, there exists a  null curve $\wh X  = ( \wh X_1, \wh X_2, \wh X_3) \colon \overline M \to \C^3$ of class  $\Ascr^1(\overline M)$
		satisfying the following:
		\begin{enumerate}[label=\rm (\roman*)]
			\item \label{1} $\| \wh X - X \|_{1, K} < \varepsilon$,
			\item \label{2} $\mgot(\wh X)>s$ on $b\overline{M}$,
			\item \label{5} $|  \wh X_1 + \igot \wh X_2   |  > s_0    \text{ on } b\overline{M}$,
			\item  \label{3} $ \| \wh X_3 - X_3 \|_{0, \overline M } < \varepsilon$,
			\item \label{4} $\wh X - X $ vanishes to order $k$ everywhere on $\Lambda$.
		\end{enumerate}
	\end{lemma}
	\begin{proof}
		The proof follows the ideas of \cite[Lemma 4]{AlarconForstneric2015MA}, but applied differently.
		Through the proof, we will denote $V_1=\frac{1}{\sqrt2}(1,i,0)$ and $V_2=\frac{1}{\sqrt2}(1,-i,0)$ (note that $V_1$ and $V_2$ are orthonormal). Hence, for a map $f$ we have that
		\begin{equation*}
			\label{mV} \mgot(f)=\max\{| \langle f,V_1\rangle|,|\langle f,V_2\rangle|\}, 
		\end{equation*}
		recall \eqref{eq:m}.
		%
		We assume without loss of generality that $\overline{M}$ is a relatively compact bordered domain in an open Riemann surface $\wh M$ such that $\overline M$ is a Runge subset of $\wh M$. By the Runge-Mergelyan theorem 	\cite[Theorem 3.6.2]{AlarconForstnericLopez2021Book}, we may assume that $X$ extends holomorphically to $\wh M$.
		%
		We also assume that $ b \overline M$ has only one connected component. In the general case we apply the same reasoning to each component. 

		By a general position argument \cite[Theorem 3.4.1 (a)]{AlarconForstnericLopez2021Book} we may assume there is a point $q\in b\overline M$ such that $X_3(q)\neq 0$. We choose a compact smooth embedded arc $\lambda\subset \C^3$ such that
		\begin{enumerate}[label= \rm (\alph*$1$)]
			\item \label{disj} one endpoint of $\lambda$ is $X(q)$ and $\lambda$ is otherwise disjoint from $X(\overline M)$, 
			\item \label{l1} $| \langle z,V_1 \rangle |>s_0/ \sqrt{2} $ for any $z\in\lambda$, see \eqref{eq:boundary},
			\item \label{end} the other endpoint of $\lambda$ is denoted $v\in \C^3$, and $|\langle v,V_1\rangle |>s$,
			\item \label{l3}  $ |  \langle z,(0,0,1)\rangle -X_3(q)|<\varepsilon/2$ for any $z\in \lambda$.
		\end{enumerate}
		Up to slightly modifying the arc $\lambda$, we may apply the method of exposing boundary points, see \cite[Theorem 6.7.1]{AlarconForstnericLopez2021Book} (see also \cite{ForstnericWold2009}), in order to approximate $X$ in the $\Cscr^1$ topology outside a small open neighborhood of $q$ by a null curve $X^0=(X^0_1,X^0_2,X^0_3)\colon \overline M\to \C^3$ of class $\Ascr^1(\overline M)$ with $X^0(q)=v$ such that
		\begin{enumerate}[label= \rm (\alph*$2$)]
			\item  \label{x0ap} $\|X^0- X\|_{1,K}<\varepsilon /2$,
			\item \label{x01} $| \langle X^0,V_1 \rangle |> s_0 / \sqrt{2}$ on $b\overline M$, recall that $s>s_0$ and see \ref{l1} and \eqref{eq:boundary},
			\item  \label{x0v1} $|\langle X^0(q),V_1\rangle|>s$, see \ref{end},
			\item \label{x03} $\|X^0_3-X_3\|_{0,\overline{M}}<\varepsilon/2$, see \ref{l3},
			\item \label{x0i} $X^0-X$ vanishes to order $k$ everywhere on $\Lambda$. 
		\end{enumerate} 
		By \ref{x0v1} and the continuity of $X^0$ there is a compact subarc $\alpha\subset b\overline M\setminus \{q\}$ with
		\begin{equation}\label{eq:C-a}
			|\langle X^0,V_1\rangle|>s\quad \text{ on $\quad {b\overline M\setminus \text{int}(\alpha)}$},
		\end{equation}
		where, from now on, $\text{int}(\alpha)$ means the relative interior of the curve $\alpha$ in $b\overline M$.
		Thus we may consider another compact subarc $\beta\subset b\overline M\setminus\{q\}$ such that $\alpha\subset \text{int}(\beta)$,
		a smooth function $\mu\colon b\overline M\to \R^+$ satisfying
		\begin{equation}\label{mu}
			\mu =0 \quad \text{ on $\quad {b\overline M\setminus \beta}$},
		\end{equation}
		and a map $\vartheta\colon b\overline M\times \overline\D\to \C^3$ given by
		\begin{equation}\label{eq:vt}
			\vartheta(p,\xi)=\begin{cases}
				X^0(p), & p\in b\overline M\setminus \beta,\\
				X^0(p)+\mu(p)\xi V_2, & p\in \beta,
			\end{cases}
		\end{equation}
		such that
		\begin{enumerate}[label= \rm (\alph*$3$)]
			\item \label{vts} $|\langle \vartheta(p,\xi),V_1\rangle|>s $ for all $p\in {b\overline M\setminus\text{int}(\alpha)}$ and $\xi\in b\D$, 
			\item \label{vt1} $|\langle \vartheta(p,\xi),V_1\rangle|>s_0/\sqrt{2}$ for all $p\in b\overline M$ and $\xi\in b\D$, 
			\item \label{vtV2} $|\langle\vartheta(p,\xi),V_2\rangle|> s $ 
			for all $p\in \alpha$ and $\xi\in b\D$.
		\end{enumerate}
		For \ref{vts} take into account \eqref{eq:C-a} and that $\langle V_1,V_2\rangle=0$; and for \ref{vt1} use \ref{x01} and $\langle V_1,V_2\rangle=0$.
		Notice that \eqref{mu} implies the continuity of $\vartheta$ and that \ref{vtV2} follows from choosing $\mu$ big enough in $\alpha$.
		Now we apply the Riemann-Hilbert theorem with jet interpolation \cite[Theorem 6.4.2]{AlarconForstnericLopez2021Book} to obtain a null curve $\wh X=(\wh X_1,\wh X_2,\wh X_3)\colon \overline M \to\C^3$ of class $\Ascr^1 (\overline M)$ such that
		\begin{enumerate}[label= \rm (\alph*$4$)]
			\item \label{hXa} $\| \wh X-X^0\|_{1,K}<\varepsilon/2$,
			\item \label{hXd} $\dist (\wh X(p),\vartheta(p,b\D))<\varepsilon$ for all $p\in b\overline M$,
			\item \label{hX3} $\|\wh X_3-X^0_3\|_{0,\overline{M}}<\varepsilon/2$, see \cite[Theorem 6.4.2, \rm ii)--iii)]{AlarconForstnericLopez2021Book},
			\item \label{hXi} $\wh X - X^0$ vanishes to order $k$ everywhere on $\Lambda$.
		\end{enumerate}
		The null curve $\wh X$ satisfies the conclusions of the lemma, provided that $\varepsilon$ is chosen small enough. Indeed,
		\ref{1} follows from \ref{x0ap} and \ref{hXa}.
		Property \ref{2} is a consequence of  \ref{vts}, \ref{vtV2}, and \ref{hXd}.
		We obtain \ref{5} from the definition of $V_1$, \ref{vt1},  and \ref{hXd}.
		Conditions  \ref{x03} and \ref{hX3} imply \ref{3}.
		Finally, \ref{4} is given by \ref{x0i} and \ref{hXi}.
	\end{proof}

	\section{Proper maps in $\mathbb{C}^2\times\mathbb{C}^*$}\label{sec:proper}
	This section is devoted to prove different Runge approximation theorems for holomorphic null curves in $\C^3$, with control on the zero set of the third component function and additionally ensuring that certain holomorphic maps, depending on the null curves, are proper or almost proper.  
	\begin{theorem}\label{th:cantor}
		Let $M$ be a compact Riemann surface, $\Omega_0$ a smoothly bounded compact connected convex domain in a holomorphic coordinate chart on $M$, and $X = (X_1, X_2, X_3) \colon M \setminus \mathring{\Omega}_0 \to \C^3$ a  holomorphic null curve.
		For any $\varepsilon >0$ there exists a Cantor set $C \subset \mathring\Omega_0$ and an injective nonflat holomorphic null curve $\wt X = (\wt X_1, \wt X_2, \wt X_3) \colon M \setminus C \to \C^3$ such that 
		\begin{enumerate} [label={\rm (\alph*)}]
			
			\item \label{aprox} $|| \wt X - X ||_{1, M\setminus \mathring{\Omega}_0} < \varepsilon$,
			
			\item \label{zeros} $\wt X_3^{-1} (0) = X_3^{-1} (0)$,
			
			\item \label{proper} $(\wt X_1, \wt X_2) \colon M \setminus C  \to \C^2$ and $(1, \wt X_1, \wt X_2)/ \wt X_3 \colon M \setminus (C \cup X_3^{-1} (0)) \to \C^3$ are proper maps. 
			\newcounter{Mayus1}\setcounter{Mayus1}{\value{enumi}}
		\end{enumerate}
	\end{theorem}
	Along the proof, we use the following inductive construction of a Cantor set in a smoothly bounded compact connected convex domain $\Omega_0 \subset \C$, as explained in \cite{Forstneric2022RMI} or in \cite[\S 6]{AlarconForstneric2023}.  First, denote by $E_0$ the vertical straight line segment within $\Omega_0$ splitting it into two halves with the same width. Slightly shrinking the closures of these  halves, we get two smoothly bounded compact connected convex subsets $\Delta_0^1$ and $\Delta_0^2$ inside $\mathring{\Omega}_0$, as in Figure \ref{fig:c}. 
	Second, denote by $E_0^j$ the horizontal straight line segments partitioning 
	$\Delta_0^j$ into two subsets of equal height for $j=1,2$.  After shrinking the closures of these four subsets, we get four smoothly bounded compact connected convex domains $\Omega_1^j, j = 1, \dots , 4$, see Figure \ref{fig:c}. Set
	\begin{equation}\label{eq:o1}
		\Omega_1 := \bigcup_{j=1}^4 \Omega_1^j \subset \mathring{\Omega}_0.
	\end{equation}
	
	This concludes the first step in the construction.
	In the second step, this procedure is iterated for every $\Omega_1^j$, resulting in four pairwise disjoint smoothly bounded compact connected convex domains inside each $\Omega_1^j$. This leads to a compact set $\Omega_2 \subset \mathring{\Omega}_1$  formed by the union of sixteen smoothly bounded compact connected convex  domains. Proceeding inductively, a descending sequence
	\begin{equation}\label{eq:o2}
		\Omega_0 \Supset \Omega_1 \Supset \Omega_2 \Supset \hdots
	\end{equation} 
	is established. For each $i \in \n \cup \{0\}$, the set $\Omega_i$  consists of the union of $4^i$ pairwise disjoint smoothly bounded compact connected convex domains. The intersection
	\begin{equation}\label{eq:o3}
		C = \bigcap_{i = 0}^{\infty} \Omega_i \subset \mathring{\Omega}_0
	\end{equation}
	is a Cantor set in $\mathring\Omega_0\subset \C$. This construction adapts for the case when $\Omega_0$ is contained in a holomorphic coordinate chart on a Riemann surface.
	\begin{figure}[hb]
		\begin{minipage}[b]{0.45\linewidth}
			\centering
			\includegraphics[width=0.75\linewidth]{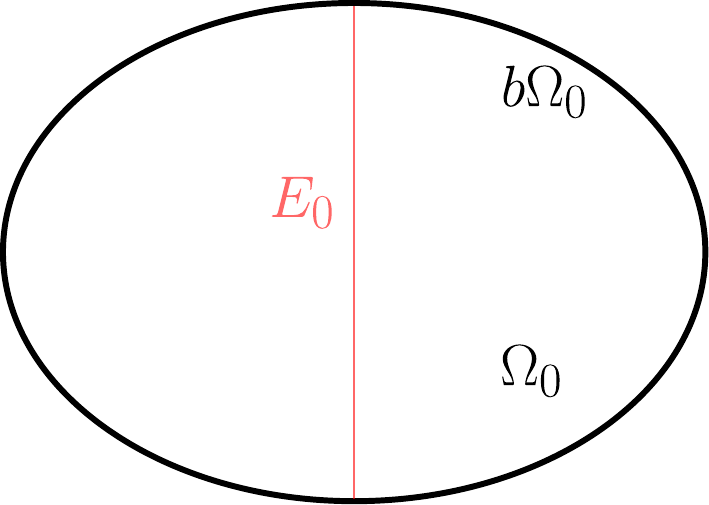} 
			\vspace{4ex}
		\end{minipage}
		\begin{minipage}[b]{0.45\linewidth}
			\centering
			\includegraphics[width=.75\linewidth]{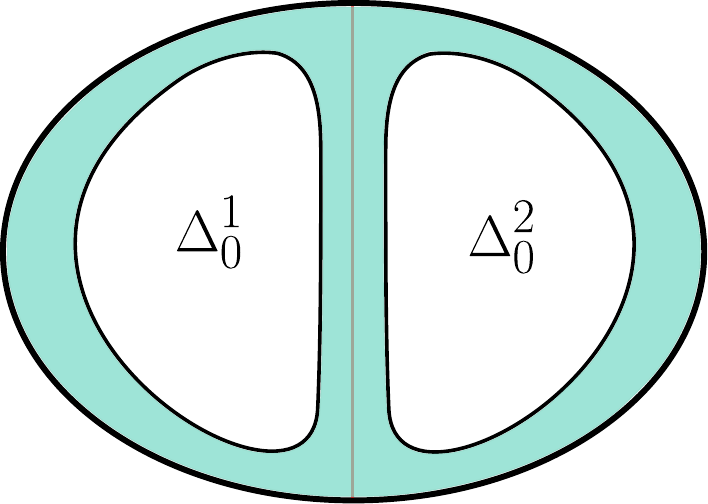} 
			\vspace{4ex}
		\end{minipage} 
		\begin{minipage}[b]{0.45\linewidth}
			\centering
			\includegraphics[width=.75\linewidth]{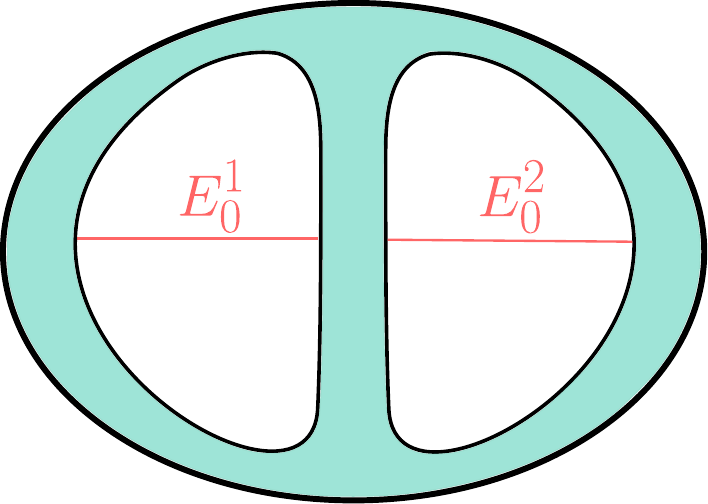} 
			\vspace{4ex}
		\end{minipage}
		\begin{minipage}[b]{0.45\linewidth}
			\centering
			\includegraphics[width=.75\linewidth]{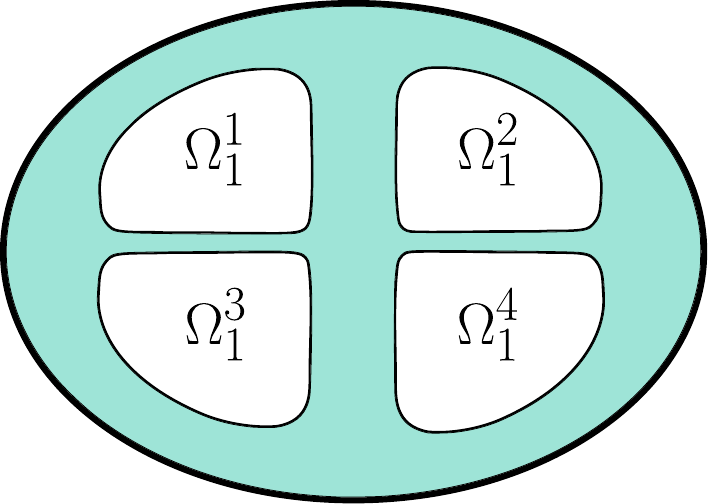} 
			\vspace{4ex}
		\end{minipage} 
		\caption{First step in the construction of a Cantor set. } \label{fig:c} 
	\end{figure}
	

	\begin{proof}[Proof of Theorem \ref{th:cantor}]
 We assume without loss of generality that $0 < \varepsilon < 1 $ and choose any $0<\varepsilon_0 < \varepsilon/2$. Set $ M_0 = M \setminus \mathring{\Omega}_0$, the latter being a smoothly bounded compact connected domain in $M$.
		We also write $X^0 = (X^0_1 , X^0_2 , X^0_3 ) :=X  \colon M_0 \to \C^3$, which we may assume to be a nonflat holomorphic null embedding and to also satisfy that
		\begin{equation}\label{eq:no0}
			X^0_3 \not = 0\quad \text{and} \quad \mgot (X^0) > 0 \text{ on } b \Omega_0 
		\end{equation}
		by a general position argument \cite[Theorem 3.4.1 (a)]{AlarconForstnericLopez2021Book} together with Hurwitz's theorem  \cite{Hurwitz1985}. Indeed, since the set $\{ (z_1, z_2, z_3) \in \C^3 \colon z_3 = 0 , z_1^2 + z_2^2 = 0\}$ is a smooth submanifold of $\C^3$ of real dimension $4$ and the boundary $b\Omega_0$ is of real dimension $1$, the general position theorem ensures \eqref{eq:no0}.
		Besides, Hurwitz's theorem guarantees that two close enough not identically zero holomorphic functions have the same number of zeros, counting multiplicity.  Hence, when we apply the general position theorem, we ask for jet interpolation on $X_3^{-1} (0)$ of high enough order, and thus we preserve the zero set of the third coordinate function provided that the approximation is close enough on $M_0$.  Finally, the compactness of $b \Omega_0$ and \eqref{eq:no0} give positive constants $c_1, c_2 \in \R$ such that
		\begin{equation}\label{eq:interval}
			0<c_1<| X^0_3| < c_2 \quad  \text{ on $ b \Omega_0.$} 
		\end{equation}
		Following the construction of a Cantor set explained before, we will construct a sequence of tuples $T_n = \{  M_n,X^n, \varepsilon_n  \}$ for any $n \in \N$, where
		\begin{itemize}
			\item $M_n = M \setminus \mathring{\Omega}_n$ and $\Omega_n$ is a compact subset in $\mathring{\Omega}_{n-1}$ consisting in $4^n$ smoothly bounded compact connected convex components, see \eqref{eq:o2}, 
			\item  $X^n \colon M_n  \to \C^3$ is a nonflat holomorphic null embedding,
			\item $\varepsilon_n > 0$ is a real number.
		\end{itemize}
		By our choice of $M_n$ and by \eqref{eq:o2} we have that 
		\begin{equation}\label{eq:exhaustion}
			M_0\Subset M_1 \Subset \cdots \Subset M_n \Subset \cdots \Subset \bigcup_{n = 0}^{\infty} M_n = M\setminus \bigcap_{n = 0}^{\infty} \Omega_n = M \setminus C,
		\end{equation}
		where $C$ is a Cantor set, see \eqref{eq:o3}. Moreover, $T_n$ will satisfy the following properties for every $n \in \N$.
		\begin{enumerate}[label=(\alph*$_n$)]
			\item \label{a} $\|X^n-X^{n-1}\|_{1,M_{n-1}}<\varepsilon_{n-1}$,
			\item \label{b} $(X_3^n)^{-1} (0) = X_3^{-1} (0)$, 
			\item \label{c} $\mgot(X^n)> n-1$ on $M_n \setminus \mathring{M}_{n-1} $,
			\item \label{d} $\mgot(X^n)>n $ on $b M_n$,
			\item \label{e} $0 < c_1 < | X_3^n| < c_2$ on $ M_n \setminus \mathring{M}_{0}$, 
			\item \label{f} $\displaystyle 0<\varepsilon_n <\varepsilon_{n-1}/2$ and every holomorphic map $Y: M_{n} \to \C^3$ satisfying $|| Y - X^{n} ||_{1, M_{n}} < 2 \varepsilon_n$ is  a nonflat embedding on $M_{n}$.
		\end{enumerate}
		Note that $T_0$ enjoy properties (\hyperref[b]{b$_0$}),  (\hyperref[d]{\rm d$_0$}), and  (\hyperref[e]{\rm e$_0$}), see \eqref{eq:no0} and \eqref{eq:interval}.
		Provided the existence of such $T_n$ for all $n \in \N$, we may consider the holomorphic map
		\begin{equation}\label{eq:lim}
			\wt X := \lim\limits_{n \to +\infty} X^n \colon M\setminus C = \bigcup_{n = 0}^{\infty} M_n \to \C^3,
		\end{equation}
		which satisfies 
		\begin{equation}\label{eq:eps}
			|| \wt  X - X^n ||_{1, M_n} < 2 \varepsilon_n < \varepsilon \quad \text{for all } n \in \N\cup \{0\}
		\end{equation}
		by \ref{a} and \ref{f}. Therefore $\wt X = (\wt X_1, \wt X_2, \wt X_3) \colon M\setminus C \to \C^3$ is an injective nonflat holomorphic null curve by \eqref{eq:exhaustion}, \eqref{eq:eps} and \ref{f}. Moreover, it satisfies property \ref{aprox} in the statement of the theorem by \eqref{eq:eps}. Property \ref{zeros} also holds by \ref{b}, Hurwitz's theorem and nonflatness of $\wt X$. 
		Let us explain \ref{proper}. It follows from \ref{c} and \eqref{eq:eps} that
		\[
		|(\wt X_1, \wt X_2) | + \varepsilon > | (X_1^n, X_2^n) | \ge \mgot (X^n) > n -1 \quad \text{on $M_n\setminus \mathring{M}_{n-1}$ for all $n \in \N$}.
		\]
		This and \eqref{eq:exhaustion} show that the image of a divergent curve in $M\setminus C$ is a divergent curve in $\C^2$, hence proving that $(\wt X_1, \wt X_2) \colon M \to \C^2$ is a proper map.
		Write $f=(1, \wt X_1, \wt X_2)/{\wt X_3} \colon M \setminus ( C \cup X_3^{-1} (0))  \to \C^3$. We have that $f$ is well defined by \ref{zeros}. Besides, \ref{c}, \ref{e}, \eqref{eq:eps}, and the assumption that $\varepsilon <1$ imply that
	\begin{equation}\label{eq::normt}
				|f|>	\Big|   \frac{1}{\wt X_3} (\wt X_1, \wt X_2 ) \Big| > \frac{| \left(X^n_1,  X^n_2 \right)|-\varepsilon}{|X^n_3|+\varepsilon} 	
			 > \frac{n-1}{c_2 + 1} -1 \quad   \text{ on  $M_n \setminus \mathring{M}_{n-1}$, $n\in\N$.}
	\end{equation}
	%
		Let us prove that $f$ is proper. Consider a divergent curve $\gamma\colon [0,1)\to M \setminus ( C \cup X_3^{-1} (0))$. Since $M$ is compact, we distinguish cases depending on whether $\gamma$ diverges to $C$ or $X_3^{-1} (0)$. 
		If $\gamma(t)\to C$ when $t\to 1$, equation \eqref{eq::normt} gives that  $|f\circ \gamma|$ diverges, see \eqref{eq:exhaustion}.
		In the case that $\gamma(t)\to X_3^{-1}(0)$ 
		 when $t\to 1$, since $|f\circ\gamma| > 1 / |\wt X_3 \circ \gamma|$, we have that $|f\circ\gamma|$ diverges.

		It only remains to explain the induction to construct such a sequence $T_n$.
		Recall that $T_0 = \{ M_0, X^0, \varepsilon_0\}$ satisfies (\hyperref[b]{b$_0$}),  (\hyperref[d]{\rm d$_0$}), and  (\hyperref[e]{\rm e$_0$}). Let us explain how to obtain $T_1$ from $T_0$.
		
		First, we use  (\hyperref[b]{b$_0$}),  (\hyperref[d]{\rm d$_0$}),  (\hyperref[e]{\rm e$_0$}), and \cite[Proposition 2.4]{Forstneric2022RMI} to extend $X^0$ to a nonflat generalized holomorphic null embedding $X^0 \colon M_0 \cup E_0 \to \C^3$ 
		such that $\mgot (X^0) > 0$ and $c_1<|X^0_3| < c_2$ on $E_0$, see Figure \ref{fig:c} and the construction of the Cantor set. 
		Observe also that $E_0$ intersects $M_0$ only at its endpoints, and these intersections are transverse by convexity of $\Omega_0$. This allows us to apply \cite[Theorem 3.1]{AlarconCastro-InfantesHidalgo2023} and approximate $X^0$ on the $\Cscr^1$ topology on $M_0 \cup E_0$ by a nonflat holomorphic null embedding $Y^0 \colon M \setminus (\mathring{\Delta}_0^1\cup \mathring{\Delta}_0^2)  \to \C^3$, where $\Delta_0^j$, $j=1,2$ are chosen sufficiently big, see Figure \ref{fig:c}, such that 
		\begin{equation}\label{eq:Y}
			\mgot (Y^0) >  0, \hspace{2mm} c_1<| Y^0_3 | < c_2   \text{ on }  M \setminus (\mathring{\Delta}_0^1\cup \mathring{\Delta}_0^2),   \hspace{2mm} \text{and}\hspace{2mm} (Y^0_3)^{-1} (0) = (X_3^0)^{-1} (0).
		\end{equation}
		In view of \eqref{eq:Y}, we can use \cite[Proposition 2.4]{Forstneric2022RMI} to extend $Y^0$ to $E_0^1 \cup E_0^2$ 
		such that $\mgot (Y^0) > 0$ and $c_1<|Y^0_3| <c_2$ on $E_0^1 \cup E_0^2$.
		Applying \cite[Theorem 3.1]{AlarconCastro-InfantesHidalgo2023} we approximate $Y^0$ in the $\Cscr^1$ topology on $ (M \setminus (\mathring{\Delta}_0^1\cup \mathring{\Delta}_0^2)  )\cup E_0^1 \cup E_0^2$ by a nonflat holomorphic null embedding $Z^0\colon M_1:= M\setminus \mathring{\Omega}_1 \to \C^3$ where $\Omega_1$ is chosen big enough, see \eqref{eq:o1} and Figure \ref{fig:c}, so that
		\begin{equation}\label{eq:Z}
			\mgot(Z^0) > 0 ,  \quad c_1< | Z^0_3 | < c_2   \text{ on } M_1, \quad \text{and}\quad
			(Z^0_3)^{-1}(0) = (Y_3^0)^{-1} (0).
		\end{equation}
		Now we apply Lemma \ref{lemma:ap}, Hurwitz's theorem, and a general position argument (see \cite[Theorem 3.4.1 (a)]{AlarconForstnericLopez2021Book}) to obtain a nonflat holomorphic null embedding  $X^1 \colon M_1 \to \C^3$ satisfiying (\hyperref[a]{\rm a$_1$}),  (\hyperref[b]{\rm b$_1$}), (\hyperref[c]{\rm c$_1$}), (\hyperref[d]{\rm d$_1$}), (\hyperref[e]{\rm e$_1$}), provided that the holomorphic null embeddings $X^0$, $Y^0$, $Z^0$, and $X^1$ are chosen sufficiently close to each other in the $\Cscr^1$ norm on $M_0$.
		We choose a positive real number $\varepsilon_1$ for which (\hyperref[f]{\rm f$_1$}) holds.
		This gives $T_1$. The inductive step to prove the existence of the sequence $T_n$ consists of an analogous reasoning.
	\end{proof}

	\begin{theorem}\label{th:bordered}
		Let $M$ be a bordered Riemann surface, $X = (X_1, X_2, X_3) \colon \overline{M} \to \C^3$ be a null curve of class $\Ascr^1(\overline{M})$. Let $\Lambda\subset M$ be a finite subset and $K\subset M$ be a smoothly bounded compact Runge domain.
		Assume also that  
		\( 
			X_3^{-1} (0) \subset M 
		\) and that there exists a number $s_0>1$ such that
		\begin{equation}\label{eq:hypo}
		|X_1+\igot X_2|>s_0 >1\quad \text{ on $\overline M$.}
	\end{equation}
		For any real number $\varepsilon>0$
		and any $k\in \n$ there exists a nonflat holomorphic null curve $\wt X = (\wt X_1, \wt X_2, \wt X_3) \colon M  \to \C^3$ such that 
		\begin{enumerate} [label={\rm (\alph*)}]
			
			\item \label{aproxb} $|| \wt X - X ||_{1, K} < \varepsilon$,
			
			\item \label{zerosb} $||\wt X_3 - X_3 ||_{0, \overline M} < \varepsilon$,
			
			\item\label{zerosb3} $\wt X_3^{-1} (0) = X_3^{-1} (0)$,
			
			\item \label{interb} $\wt X-X$ vanishes up to order $k$ everywhere on $\Lambda$,
			
			\item \label{properb} $(1, \wt X_1, \wt X_2)/ \wt X_3 \colon M \setminus X_3^{-1} (0) \to \C^3$ and  $h\colon M\setminus  X_3^{-1}(0)\to \R^+$ given by 
			\[
			\frac{1}{| \wt X_3|^4 } \left( \left( \left|  \wt X_1 + \igot  \wt X_2 \right|^2 -1    \right)^2 +  \left(  \left| \wt X_1^2 + \wt X_2^2 + \wt X_3^2\right|^2  -  \left| \wt X_1 - \igot \wt  X_2\right|^2  \right)^2      \right)
			\]
			are almost proper maps.
		\end{enumerate}
		Moreover, when $X|_\Lambda$ is injective, then $\wt X$ may be chosen injective too.
	\end{theorem}
	\begin{proof} 
		Assume that $X|_\Lambda$ is injective, otherwise the proof is simpler. 
		Choose a number $\eta > 0$ such that 
		\begin{equation}\label{eq:eta}
			|| X_3||_{0, \overline{M}} + \varepsilon < \eta.
		\end{equation}
		Set $X^0 =  (X^0_1, X_2^0, X_3^0) : = X$ and $\varepsilon_0 : = \varepsilon$. 
		By the general position theorem \cite[Theorem 3.4.1 (a)]{AlarconForstnericLopez2021Book} we may assume that $X^0$ is a nonflat null embedding of class $\Ascr^1 (\overline M)$.  
		Set
		\begin{equation}\label{eq:sn}
			s_n \colon = s_0 + n
		\end{equation}
		and choose $M_0\subset M$ a smoothly bounded compact connected Runge domain with
		\begin{equation}\label{eq:k}
			K\cup X_3^{-1} (0)\subset \mathring M_0.
		\end{equation}
		We shall inductively construct a sequence $\{M_n,X^n,\varepsilon_{n}\}$ for any $n \in \N$ of smoothly bounded compact connected Runge domains $M_n$,
	 	null embeddings $X^n = (X_1^n, X_2^n, X_3^n) \colon \overline M  \to \C^3$ of class $\Ascr^1(\overline{M})$, and constants $\varepsilon_n>0$, such that
	\begin{equation}\label{eq:ex}
		M_0 \Subset M_1 \Subset \cdots \Subset M_n \Subset \cdots \Subset \bigcup_{n = 0}^{\infty} M_n = M,
	\end{equation} 
	and such that the following properties hold for every $n \in \N $.
		\begin{enumerate}[label=({\alph*$_n$})]
			\item \label{seq:aprox} $\|X^n-X^{n-1}\|_{1,M_{n-1}}<\varepsilon_{n-1}$,
			\item \label{seq:aprox3} $|| X_3^n - X_3^{n-1}||_{0, \overline{M}} < \varepsilon_{n-1}$,
			\item \label{seq:zeros3} $(X_3^n)^{-1} (0) = X_3^{-1} (0)$, 
			\item \label{seq:inter} $X^n-X$ vanishes to order $k$ everywhere on $\Lambda$,
			\item \label{seq:mgot} $\mgot(X^n)>s_n $ on $b M_n $,
			\item \label{seq:raiz2} $| X_1^n + \igot X_2^n | > s_0$ on $\overline M\setminus \mathring M_n $,
			\item \label{seq:embeded} $\displaystyle 0<\varepsilon_n <\varepsilon_{n-1}/2$ and every holomorphic map $Y\colon M_{n} \to \C^3$ satisfying $|| Y - X^{n} ||_{1, M_{n}} < 2 \varepsilon_n$ is  a nonflat embedding on $M_{n}$.
		\end{enumerate}
	
		Let us first prove the basis of the induction. By \eqref{eq:hypo} we can apply Lemma \ref{lemma:apds} to obtain a null curve $X^1 \colon \overline M \to \C^3$ of class $\Ascr^1 (\overline M)$ satisfying properties  \hyperref[seq:aprox]{(a$_1$)}, \hyperref[seq:aprox3]{(b$_1$)}, and \hyperref[seq:mgot]{(d$_1$)} as a direct consequence of \ref{1},  \ref{3} and \ref{4}.  Property  \hyperref[seq:zeros3]{(c$_1$)} is implied by  \hyperref[seq:aprox3]{(b$_1$)} and Hurwitz's theorem, provided interpolation on the set $X_3^{-1}(0)$ of high enough order.
		Moreover, by continuity of $X^1$, choosing a sufficiently big smoothly bounded compact connected Runge domain $M_1 \subset M$, properties \hyperref[seq:raiz2]{(e$_1$)} and  \hyperref[seq:raiz2]{(f$_1$)}  hold by  \ref{2} and
		\ref{5} of Lemma \ref{lemma:apds}.
		By Hurwitz's theorem and a general position argument we may assume that $X^1$ is a nonflat null embedding of class $\Ascr^1 (\overline M)$, recall that $X$ is injective on $\Lambda$.
		Finally we choose a positive number $\varepsilon_1>0$ such that  \hyperref[seq:embeded]{(g$_1$)} holds.  
		
		For the inductive step, in view of \ref{seq:raiz2} we can apply Lemma \ref{lemma:apds} to the data
		\[
		X^{n-1} \colon \overline M \to \C^3, \quad M_{n-1}, \quad \Lambda\cup X_3^{-1}(0), \quad s_n>s_0,  \quad\text{and}\quad \varepsilon_{n-1},
		\]
		thus obtaining a null curve $X^n \colon \overline M \to \C^3$ of class $\Ascr^1 (\overline M)$. Next, we choose $M_n$ and $\varepsilon_n$ arguing as in the basis case of induction. In this way we construct a sequence $\{ M_n, X^n, \varepsilon_n\}_{n \in \n}$ satisfying  \ref{seq:aprox}--\ref{seq:embeded}.
		
		By  \eqref{eq:ex}, \ref{seq:aprox} and \ref{seq:embeded}, the sequence $\{X^n\}_{n \in \n}$ converges uniformly on compacts in $M$ to a holomorphic map
		\[
		\wt X : = \lim\limits_{n \to +\infty} X^n \colon M \to \C^3
		\]
		satisfying 
		\begin{equation}\label{eq::n0}
			\|\wt X-X^n\|_{1, M_n}< 2\varepsilon_{n}<\varepsilon\quad \text{for all }n\in \N\cup \{0\}.
		\end{equation}
		Therefore, $\wt X=(\wt X_1,\wt X_2,\wt X_3)\colon M\to\C^3$ is a nonflat injective holomorphic null curve by \ref{seq:embeded},  \eqref{eq:ex} and \eqref{eq::n0}. Let us show that $\wt X$ satisfies the conclusions of the theorem. Property \ref{aproxb} follows from \eqref{eq:k} and \eqref{eq::n0}. 
		Property  \ref{zerosb} follows from \ref{seq:aprox3}. Property \ref{zerosb3} follows from \ref{seq:zeros3}, Hurwitz's theorem, and nonflatness of $\wt X$.
		Finally, \ref{interb} follows from \ref{seq:inter}.

		It remains to prove \ref{properb}; note that both maps in its statement are well defined by \ref{zerosb}. 
			Write $f=(1, \wt X_1, \wt X_2)/{\wt X_3} \colon M \setminus  X_3^{-1} (0)  \to \C^3$. Properties \eqref{eq:eta}, \eqref{eq:k}, 
			\ref{zerosb}, 
			\ref{seq:mgot} , and  \eqref{eq::n0} imply that
		\[ 
			|f|>	\Big|   \frac{1}{\wt X_3} (\wt X_1, \wt X_2 ) \Big| > \frac{| \left(X^n_1,  X^n_2 \right)|-\varepsilon}{|X_3|+\varepsilon} 	
			> \frac{s_n}{\eta} -1 \quad   \text{ on $b M_n$.}
		\]
		Recall that $\{ s_n \}_{n \in \n}$ diverges by \eqref{eq:sn}. Then, this inequality, \eqref{eq:ex}, and the fact that $1/ \wt X_3 \colon W \setminus X_3^{-1} (0) \to \C$ is a proper map, where $W$ is any small closed neighbourhood of $X_3^{-1} (0)$ in $M$, imply that $f$ is an almost proper map. 
		Let us now show that $h$ is also an almost proper map. 
		Choose for each $n \in \n $  a compact domain $W_n\subset \mathring M_0$ which is a union of pairwise disjoint smoothly bounded closed discs $W_{n,p}$ for any $p \in X_3^{-1} (0)$, such that $p \in \mathring{W}_{n,p} \Subset M_n$, recall \eqref{eq:k} and \eqref{eq:ex}.
		Set $M_n' = M_n \setminus \mathring{W}_n$, which is a smoothly bounded compact connected domain with $b M_n ' = b M_n \cup b W_n$.
		By \ref{zerosb}, \ref{zerosb3} and \eqref{eq:ex}, we may choose the discs $W_{n,p}$ shrinking at each step such that 
		\begin{equation}\label{eq:Mn''}
			M_0' \Subset M_1' \Subset \cdots \subset  \bigcup_{n \in \n} M_n'  = M \setminus X_3^{-1} (0) =M \setminus \wt X_3^{-1} (0), 
		\end{equation}
		and satisfying
		\begin{equation}\label{eq:boundWn}
			| \wt X_3 |^4 < 1/ s_n \quad \text{on  $b W_n$ for all $n\in\N $.} 
		\end{equation}
		We are going to show that  $h> (s_0 -1)^2 \, s_n$ on $b M_n'$ for every large enough $n \in \n$. Since $s_0 >1$ by \eqref{eq:hypo}, $\{ s_n \}_{n \in \n}$ diverges by \eqref{eq:sn}, and by \eqref{eq:Mn''} this would imply almost properness of $h$. 
		From \eqref{eq:hypo} and \eqref{eq::n0}, we have that
		\begin{equation}\label{eq:bound2}
			|  \wt X_1 + \igot  \wt X_2 |^2-1  > \left( \left|  X_1^0  + \igot   X_2^0 \right| - \varepsilon \right)^2 -1 > (s_0 - \varepsilon)^2 -1 > s_0 - 1 > 0 \quad \text{on } M_0, 
		\end{equation}
	provided that $\varepsilon< s_0 - \sqrt{s_0}$. 
		Since $b W_n \subset \mathring{M}_0$  for all $n\in \n$, by \eqref{eq:boundWn} and \eqref{eq:bound2} it follows that 
		\begin{equation}\label{eq:boundw}
		h > 	\frac{1}{| \wt X_3|^4 }  \left( \big|  \wt X_1 + \igot  \wt X_2 \big|^2     -1  \right)^2 
		>  (s_0 -1)^2 \, s_n \quad \text{ on } b W_n, \text{ for all } n \in \n .
		\end{equation}
		Let us now prove this inequality on $b M_n$ for $n\in \n$ sufficiently large. Observe that \eqref{eq:eta} and \ref{zerosb} imply that $|| \wt X_3||_{0, M} < \eta$, hence $1/|\wt X_3| > 1/\eta$ on $M \setminus X_3^{-1} (0)$. 
		Fix any $n \in \n$ and any point $p \in b M_n$.
		Suppose that 
		\begin{equation}\label{eq:case1}
			\mgot (X^n)(p) = \frac{1}{\sqrt2} \left|X_1^n(p) + \igot X_2^n(p) \right|   > s_n,
		\end{equation}
		recall \eqref{eq:m} and \ref{seq:mgot}. By \eqref{eq::n0} and \eqref{eq:case1}, we obtain that
		\[
		h (p) > \frac{1}{|\wt X_3 (p) |^4} 
		 \left( |  \wt X_1(p) + \igot  \wt X_2(p) |^2     -1  \right)^2 > \eta^{-4} \left( ( s_n \sqrt{2} - \varepsilon)^2 -1   \right)^2.
		\]
		By \eqref{eq:sn}, for every large enough $n \in \n$ we have that $h(p)> (s_0 -1)^2  \, s_n$.
		On the other hand, when \eqref{eq:case1} does not hold, we have 
		\begin{equation}\label{eq:case2}
			\mgot (X^n)(p) = \frac{1}{\sqrt2} \left|X_1^n(p) - \igot X_2^n(p) \right|   > s_n.
		\end{equation}
		Recall that $p \in b M_n$. Then \ref{seq:raiz2}, \eqref{eq::n0}, \eqref{eq:case2}, and the triangle inequality give that
		\begin{eqnarray*}
			h(p) & > & \frac{1}{|\wt X_3 (p) |^4} 
			 \left( \left| \wt X_1 (p)^2 + \wt X_2 (p)^2 + \wt X_3 (p)^2\right|^2   -  \left| \wt X_1 (p) - \igot \wt  X_2 (p) \right|^2  \right)^2 \\
			 &= & \eta^{-4} \left| \wt X_1(p)-\igot \wt  X_2(p) \right|^4\left( \left| \wt X_1 (p) +\igot  \wt X_2 (p) + \frac{\wt X_3(p)^2 }{\wt X_1 (p) - \igot \wt X_2 (p) }\right|^2 - 1  \right)^2  \\
			&  >& \eta^{-4} \left( s_n  \sqrt{2} - \varepsilon  \right)^4  \left(  \left(  |\wt X_1 (p) + \igot  \wt X_2 (p) | - \left|  \frac{\wt X_3 (p)^2 }{\wt X_1 (p) - \igot \wt X_2 (p) }\right|   \right)^2 -1       \right)^2 \\
			&>& \eta^{-4} \left( s_n  \sqrt{2} - {\varepsilon}  \right)^4 \left(  \left( s_0 - \varepsilon -  \frac{\eta^2}{s_n \sqrt{2} - \varepsilon } \right)^2 - 1  \right)^2.
		\end{eqnarray*}
	By \eqref{eq:hypo}, \eqref{eq:sn}, our choice of $\varepsilon$, and provided that $n \in \n$ is sufficiently large, we have that $h(p)> (s_0 -1)^2 \, s_n$. 
		Hence, we showed that $h > (s_0 -1)^2  \, s_n$ on $ b M_n ' = b W_n \cup b M_n$ for every large enough $ n \in \n$. This, \eqref{eq:hypo}, \eqref{eq:sn}, and \eqref{eq:Mn''} show that \ref{properb} holds.
	\end{proof}
We finish this section with the following analogue to Theorem \ref{th:cantor}.
	\begin{theorem}\label{th:cantor2}
		Let $M$ be a compact Riemann surface, $\Omega_0$ a smoothly bounded compact connected convex domain in a holomorphic coordinate chart on $M$, and $X = (X_1, X_2, X_3) \colon M \setminus \mathring{\Omega}_0 \to \C^3$ a  holomorphic null curve. Assume there exists a number $s_0>1$ such that
		\begin{equation}\label{eq:hypo2}
			|X_1+\igot X_2| > s_0 >1  \quad \text{ on $ M\setminus \mathring \Omega_0$.}
		\end{equation}
		For any $\varepsilon >0$ there exists a Cantor set $C \subset \mathring\Omega_0$ and an injective nonflat holomorphic null curve $\wt X = (\wt X_1, \wt X_2, \wt X_3) \colon M \setminus C \to \C^3$ such that 
		\begin{enumerate} [label={\rm (\alph*)}]
			
			\item \label{aprox2} $|| \wt X - X ||_{1, M\setminus \mathring{\Omega}_0} < \varepsilon$,
			
			\item \label{zeros2} $\wt X_3^{-1} (0) = X_3^{-1} (0)$.
			
			\item \label{almostproper} $(1, \wt X_1, \wt X_2)/ \wt X_3 \colon M \setminus (C \cup X_3^{-1} (0)) \to \C^3$ and $h\colon M\setminus (C\cup  X_3^{-1}(0))\to \R^+$ given by 
			\[
			\frac{1}{| \wt X_3|^4 } \left( \left( \left| \wt X_1 + \igot  \wt X_2 \right|^2  -1   \right)^2 +  \left(  \left| \wt X_1^2 + \wt X_2^2 + \wt X_3^2\right|^2  - \left| \wt X_1 - \igot  \wt X_2\right|^2   \right)^2  \right)
			\]
			are almost proper maps.
		\end{enumerate}
	\end{theorem}
	\begin{proof}
		We shall follow the construction of the Cantor set and the notation $M_0$, $X^0$ and $\varepsilon_0$ in the proof of Theorem \ref{th:cantor}, and assume that $X^0$ is a null embedding that satisfies \eqref{eq:no0} and \eqref{eq:interval}. 
		Assume also that $\varepsilon< s_0 - \sqrt{s_0}$.
		We will again construct a sequence  of tuples $T_n = \{ M_n, X^n , {\varepsilon}_n\}$ for any $n \in \n$ satisfying the properties \eqref{eq:exhaustion},  \ref{a}, \ref{b}, \ref{d}, \ref{e}, and \ref{f} in the proof of Theorem  \ref{th:cantor}. However, the following property will hold instead of \ref{c}:
		\begin{enumerate}[label=(\alph*$_n'$)]
			\setcounter{enumi}{2}
			\item  \label{c'} $|  X_1^n + \igot  X_2^n | > s_0 > 1 $ on $bM_n$.
		\end{enumerate}
		This will lead to a Cantor set $C \subset \mathring{\Omega}_0$ as in \eqref{eq:exhaustion} (not necessarily the same) and to a limit holomorphic null curve
		\[
		\wt X \colon = \lim\limits_{n \to + \infty}  X^n \colon M \setminus  C \to \C^3,
		\]
		which is well defined and satisfies \ref{aprox2},  \ref{zeros2}, see the proof of Theorem \ref{th:cantor}. 
		Condition \ref{almostproper} follows by arguing as in the proof of Theorem \ref{th:bordered}. Note that \ref{b} ensures that $X_3^{-1} (0) \subset M \setminus {\Omega}_0=\mathring M_0$, that is, equation \eqref{eq:k} holds. Then, the first map in   \ref{almostproper} is almost proper
 by	\eqref{eq:exhaustion}, \ref{d}, \ref{e}, \eqref{eq:eps},  and \eqref{eq:k}. The map $h$ is almost proper by
		 \ref{a}, \ref{b},  \ref{d}, \ref{f},  \eqref{eq:hypo2}, and (\hyperref[c']{\rm c$_n'$}). 

		It only remains to explain the inductive step to construct such sequence $T_n$.
		Let us explain how to obtain $T_1$ from $T_0$.
		The inductive step is again a repetition of this.
		Taking into account that $T_0$ satisfies  (\hyperref[b]{b$_0$}), (\hyperref[c']{c$_0'$}), and (\hyperref[e]{e$_0$}), we may reason as in Theorem \ref{th:cantor}, but adding
		$|Y^0_1+\igot Y^0_2|> s_0$ and $|Z^0_1+\igot Z^0_2|>s_0$ in equations \eqref{eq:Y} and \eqref{eq:Z}, respectively.
		Then we apply Lemma \ref{lemma:apds} to $ Z^0$ together with Hurwitz's theorem, and a general position argument (see \cite[Theorem 3.4.1 (a)]{AlarconForstnericLopez2021Book}). After choosing a suitable smoothly bounded compact connected domain $M_1$, this gives a nonflat holomorphic null embedding $ X^1$ satisfying (\hyperref[a]{\rm a$_1$}),  (\hyperref[b]{\rm b$_1$}), (\hyperref[c']{\rm c$_1'$}), (\hyperref[d]{\rm d$_1$}), (\hyperref[e]{\rm e$_1$}), provided that the holomorphic null embeddings $X^0$, $Y^0$, $Z^0$, and $X^1$ are chosen sufficiently close to each other in the $\Cscr^1$ norm on $M_0$. Finally, we choose ${\varepsilon}_1$ such that (\hyperref[f]{\rm f$_1$}) holds. This gives $ T_1$ as desired. 
	\end{proof}
	\begin{remark}\label{rmk:bd}
	Theorem \ref{th:cantor} and Theorem \ref{th:cantor2} also hold if $M$ is a bordered Riemann surface, recall Definition \ref{def:bordered}. Proving this is just a repetition of the corresponding proofs but changing the exhaustion in \eqref{eq:exhaustion} by one of the form $ M_n' \setminus \mathring{\Omega}_n$, where $\{ M_n'\}_{n \in \N \cup \{0\}}$ is an exhaustion of $M$ with $\Omega_0 \subset \mathring{M}_0 '$.

\end{remark}
\begin{remark}
	If we have a finite subset $\Lambda \subset M \setminus \Omega_0$ and any $k \in \n$ in the hypothesis of Theorem \ref{th:cantor} and \ref{th:cantor2}, we can ensure that
	\begin{itemize}
		\item[(d)] $\wt X - X$ vanishes up to order $k$ everywhere on $\Lambda$,
	\end{itemize}
by applying the corresponding interpolation properties in Lemma \ref{lemma:ap} or Lemma \ref{lemma:apds}.
\end{remark}
\begin{remark}
	The assumptions in \eqref{eq:hypo} and \eqref{eq:hypo2} are technical conditions for our approach to work, but we expect they are not necessary for the statements of Theorems  \ref{th:bordered} and \ref{th:cantor2}. 
\end{remark}


	\section{From holomorphic null curves to $\CMC$ surfaces in $\mathbb{H}^3$ and $\mathbb{S}^3_1$}\label{sec:HS} 
	We devote this section to the proofs of the theorems stated in the introduction, which are a consequence of the results proved in Section \ref{sec:proper}. 
	Let us first recall some properties and notation.

Recall that the special linear group $\SL \subset \Mcal_2(\C)$, as a subset of the $2\times 2$ matrices with complex entries, is 
\[
\SL=\left\lbrace A=\left(\begin{matrix}
	z_{11} & z_{12}\\
	z_{21} & z_{22}
\end{matrix}\right) \in \Mcal_2(\C): \det A=z_{11}z_{22}-z_{12}z_{21}=1 \right\rbrace .
\]
Martín, Umehara and Yamada \cite{MartinUmeharaYamada2009CVPDE} showed that the biholomorphism $\Tcal\colon \C^2\times\C^*\to \SL\setminus \{z_{11}=0\}$ given by
\begin{equation}\label{eq:Tcal}
	\Tcal(z_1,z_2,z_3)=\frac{1}{z_3}\left(\begin{matrix}
		1 & z_1+\igot z_2\\
		z_1-\igot z_2 & z_1^2+z_2^2+z_3^2
	\end{matrix}\right),\quad (z_1,z_2,z_3)\in \C^2\times\C^*,
\end{equation}
takes holomorphic null curves in $ \C^2\times\C^*$ into holomorphic null curves in $\SL\setminus\{z_{11}=0\}$. Further, the inverse $\Tcal^{-1}$ takes holomorphic null curves in $\SL\setminus\{z_{11}=0\}$ into holomorphic null curves in  $ \C^2\times\C^*$.
Given an open Riemann surface $M$, a {\em holomorphic null curve in} $\SL$ is a holomorphic immersion $F\colon M\to\SL$ which is directed by the quadric
\[
\Acal=\left\lbrace\left(\begin{matrix}
	z_{11} & z_{12}\\
	z_{21} & z_{22}
\end{matrix}\right) \in \Mcal_2(\C): z_{11}z_{22}-z_{12}z_{21}=0 \right\rbrace,
\]
that is, the derivative $F'\colon M\to\Mcal_2(\C)$ with respect to any local holomorphic coordinate on $M$ assumes values in $\Acal\setminus\{0\}$. 

Let $\mathbb{L}^4$ be the Lorentz-Minkowski space of dimension $4$ with the canonical Lorentzian metric $-dx_0^2+dx_1^2+dx_2^2+dx_3^2$. We consider the following identification with the hermitian matrices $\text{Her}(2)\subset \Mcal_2(\C)$:
\begin{equation}\label{eq:her}
	(x_0,x_1,x_2,x_3)\in \mathbb{L}^4 \Longleftrightarrow 
	\left(\begin{array}{cc}
		x_0 + x_3 & x_1+\igot x_2\\
		x_1-\igot x_2 & x_0 - x_3
	\end{array}\right)\in \text{Her}(2).
\end{equation}
The hyperboloid model for the hyperbolic $3$-space $\h^3$ is given by
\[
\h^3=\left\lbrace (x_0,x_1,x_2,x_3)\in \mathbb{L}^4 : -x_0^2+x_1^2+x_2^2+x_3^2=-1, x_0>0 \right\rbrace
\]
endowed with the metric induced by $\mathbb{L}^4$. With the above identification
\[
\h^3=\left\lbrace A \overline{A}^T : A\in \text{SL}_2 (\C) \right\rbrace=\text{SL}_2(\C)/ \text{SU}(2);
\]
here $\overline{\cdot}$ and $\cdot^T$ mean complex conjugation and transpose matrix, respectively, and $\text{SU}(2)$ is the special unitary group. 
In this setting, the canonical projection
\begin{equation*}\label{eq:piH}
	\pi_H\colon \text{SL}_2(\C)\to \h^3,\quad \pi_H(A)=A \overline{A}^T,
\end{equation*}
maps holomorphic null curves in $\text{SL}_2(\C)$ to conformal $\CMC$ immersions, or equivalently, Bryant surfaces, see \cite{Bryant1987Asterisque}. Since $\text{SU}(2)$ is compact,  $\pi_H$ is a proper map and hence it takes proper holomorphic null curves in $\text{SL}_2(\C)$ to proper conformal $\CMC$ immersions in $\h^3$. 
%
It should be noted that the restriction of $\pi_H$ to $\SL \setminus
\{z_{11} = 0\}$ is still surjective, and that the projection $\pi_H \circ \Tcal \colon \C^2 \times \C^* \to \h^3$ takes holomorphic null curves into conformal $\CMC$ immersions. 
	\begin{proof}[Proof of Theorem \ref{th:mainH}]
		Let $\Omega_0$ be a smoothly bounded  compact connected convex domain in a holomorphic coordinate chart on $M$. Consider any holomorphic null curve $X=(X_1,X_2,X_3)\colon M\setminus \mathring \Omega_0\to\c^3$ such that $X_3^{-1}(0)=\varnothing$, whose existence follows from \cite[Theorem 3.6]{AlarconForstnericLopez2021Book} and a suitable translation.
		
		Now we apply Theorem \ref{th:cantor} to $X$, which provides a Cantor set $C\subset M$ and a holomorphic null curve $\wt X = (\wt X_1, \wt X_2, \wt X_3) \colon M\setminus C\to \C^3$ such that 
		\begin{enumerate}[label=(\alph*1)]
			\item \label{eq:properH} $(1, \wt X_1, \wt X_2)/ \wt X_3 \colon M \setminus C  \to \C^3$ is a proper map,
			\item \label{x3zc}$\wt X_3^{-1} (0) = X_3^{-1}(0) = \varnothing$.
		\end{enumerate}
		We claim that the following map solves the theorem:
		\[
		\varphi :=\pi_H\circ \Tcal \circ \wt X\colon M\setminus C\to\h^3.
		\]
		By the preceding discussion and \ref{x3zc} we have that $\varphi$ is a well defined conformal $\CMC$ immersion. Further, \ref{eq:properH} implies that $\Tcal\circ \wt X$ is proper. Indeed, if we denote $(\wt F_{ij}) = \Tcal \circ \wt X$, by \eqref{eq:Tcal} we have that
		\begin{equation}\label{eq:norm}
		|\wt F_{11}|^2 + |\wt F_{12}|^2 + |\wt F_{21}|^2 = \left(2 \left( |\wt X_1|^2 + |\wt X_2|^2 \right) + 1 \right) / | \wt X_3 |^2 .
	\end{equation}
		Properness of $\pi_H$ imply that $\varphi$ is a proper map too. 
	\end{proof}
Now we consider the de Sitter $3$-space
\[
\s_1^3=\left\lbrace (x_0,x_1,x_2,x_3)\in \mathbb{L}^4 : -x_0^2+x_1^2+x_2^2+x_3^2=1 \right\rbrace
\]
endowed with the metric induced by $\L^4$. 
Under the identification given in \eqref{eq:her}, it becomes
\[
\s_1^3=\left\lbrace A e_3 \overline{A}^T : A\in \text{SL}_2 (\C) \right\rbrace=\SL/\text{SU$_{1,1}$}, \quad \text{where } e_3 = \left(\begin{matrix}
	1&0\\
	0&-1
\end{matrix}\right),
\]
see \cite{FRUYY09}. The standard projection is given by
\begin{equation*}\label{eq:piS}
	\pi_S\colon \text{SL}_2(\C)\to \s_1^3,\quad \pi_S(A)=A e_3 \overline{A}^T.
\end{equation*}
It maps holomorphic null curves $F = (F_{ij}) \colon M \to \SL$ into $\CMC$ faces, provided that the symmetric $(0,2)$-tensor
\begin{equation}\label{eq:condps}
-	\det \left[ d \left( Fe_3 \overline{F}^T \right)   \right] = (1 - | g |^2 )^2 \omega \overline{\omega} \not \equiv 0,
\end{equation}
or equivalently, $| g|$ is not identically one, where
$
g = -   d F_{12}/d F_{11} $ is a meromorphic function on $M$, 
and $\omega$ is a holomorphic $1$-form on $M$, see \cite[Theorems 1.7 and 1.9]{Fujimori06}. Note that the projection $\pi_S$ is not proper, since $SU_{1,1}$ is not a compact subgroup of $\SL$.
%
Conversely, every simply connected $\CMC$ face in $\s_1^3$ lifts to a holomorphic null curve in $\SL$ satisfying \eqref{eq:condps}, see \cite[Theorem 1.9]{Fujimori06}. 
As before, the restriction of $\pi_S$ to $\SL \setminus
\{z_{11} = 0\}$ is still surjective, and the projection $\pi_S \circ \Tcal \colon \C^2 \times \C^* \to \s_1^3$ carries holomorphic null curves $X \colon M \to \C^2 \times \C^*$ with $\Tcal \circ X$ satisfying \eqref{eq:condps} into $\CMC$ faces.  
Finally, if a holomorphic null curve $F \colon M \to \SL$ is of the form $F = \Tcal \circ X$ for some holomorphic null curve $ X = (X_1, X_2, X_3) \colon M \to \C^2 \times \C^*$, a straightforward computation shows that
	\begin{equation}\label{eq:defg}
		g = \frac{- X_3 ' X_3}{X_1 ' - \igot X_2 '} - ( X_1 + \igot X_2).
	\end{equation}
	\begin{proof}[Proof of Theorem \ref{th:bds}]
	 Let $X=(X_1,X_2,X_3)\colon M\to\c^3$ be any holomorphic null curve with $X_3^{-1}(0)=\varnothing$, satisfying \eqref{eq:hypo}, and such that the holomorphic null curve $ \Tcal \circ X \colon M \to \SL \setminus \{ z_{11} = 0\}$ satisfies \eqref{eq:condps}; in particular, by \eqref{eq:defg} there is a point $p \in M$ such that 
	 \begin{equation}\label{eq:g}
	 	|g (p)|  =  \left| \frac{ X_3 ' (p) X_3 (p) }{X_1 ' (p) - \igot X_2 ' (p) } + ( X_1 (p) + \igot X_2 (p)) \right| \not = 1 .
	 \end{equation}
 Such a null curve may be obtained by lifting a simply connected $\CMC$ face, applying \cite[Theorem 3.1]{AlarconCastro-InfantesHidalgo2023}, and composing it with a suitable translation so that $X$ satisfies \eqref{eq:hypo} and $\Tcal \circ X$ satisfies \eqref{eq:condps}.
		%
		%
		Then, Theorem \ref{th:bordered} gives a nonflat holomorphic null curve $\wt X\colon M\to \c^3$ such that
		\begin{enumerate}[label=(\alph*2)]
			\item \label{X3z} $\wt X_3^{-1}(0)=X_3^{-1}(0)=\varnothing$,
			\item \label{int} $\wt X -  X$ vanishes to order $1$ at $p$,
			\item \label{eq:ap} the map  $(1, \wt X_1, \wt X_2)/ \wt X_3 \colon M  \to \C^3$ is almost proper,
			\item \label{h} the map $h\colon M\to \R$ given by 
			\[
			\frac{1}{| \wt X_3|^4 } \left( \left( \left|  \wt X_1 + \igot  \wt X_2 \right|^2  -  1 \right)^2 +  \left(  \left| \wt X_1^2 + \wt X_2^2 + \wt X_3^2\right|^2 - \left| \wt X_1 - \igot \wt  X_2\right|^2   \right)^2  \right)
			\]
			is almost proper.
		\end{enumerate}
		The following map solves the theorem:
	\[
			\psi:=\pi_S\circ\Tcal\circ\wt X\colon M\to\s_1^3.
	\]
		Indeed,  it is well defined by \ref{X3z}, and it is a $\CMC$ face by \eqref{eq:condps}, \eqref{eq:defg},  \eqref{eq:g}, and \ref{int}.
		A straightforward computation shows that 
		\[
			| \psi |^2 \ge 
				\frac{1}{2 | \wt X_3|^4 } \left( \left( \left| \wt X_1 + \igot  \wt X_2 \right|^2  -1   \right)^2 
				+
				\left(  \left| \wt X_1^2 + \wt X_2^2 +\wt X_3^2\right|^2  - \left| \wt X_1 - \igot \wt X_2\right|^2 \right)^2  \right) = \frac{1}{2} h,
		\]
		where $|\psi|$ denotes the euclidean norm of $\psi$ in $\r^4$.  Since $h$ is almost proper by  \ref{h}, $|\psi|$ is also almost proper and hence so is $\psi$. Finally, since $\Tcal \circ \wt X$ is a complete holomorphic null curve by \eqref{eq:norm} and \ref{eq:ap}, $\psi$ is weakly complete, see  \cite{Yu97}.
\end{proof}
	
	\begin{proof}[Proof of Theorem \ref{th:mainS}]
		 Let $X=(X_1,X_2,X_3)\colon M \setminus \mathring{\Omega}_0 \to\c^3$ be any holomorphic null curve with $X_3^{-1}(0)=\varnothing$, satisfying \eqref{eq:hypo2}, and such that the holomorphic null curve $ \Tcal \circ X \colon M \setminus \mathring{\Omega}_0  \to \SL \setminus \{ z_{11} = 0\}$ satisfies \eqref{eq:condps}; in particular, by \eqref{eq:defg} there is a point $p \in M$ such that \eqref{eq:g} holds.
	The existence of such a null curve follows reasoning as in the proof of Theorem \ref{th:bds}.
	Then we apply Theorem \ref{th:cantor2} to obtain a Cantor set $C \subset \mathring{\Omega}_0$ and a holomorphic null curve $\wt X \colon M \setminus C \to \C^3$ such that
		\begin{enumerate}[label=(\alph*3)]
				\item \label{X3z1} $\wt X_3^{-1}(0)=X_3^{-1}(0)=\varnothing$,
				\item \label{int1} $\wt X -  X$ vanishes to order $1$ at $p$,
				\item \label{eq:apS}  the map  $(1, \wt X_1, \wt X_2)/ \wt X_3 \colon M \setminus C  \to \C^3$ is almost proper,
			\item \label{eq:properS} the map $h\colon M\setminus C \to \R$ given by 
			\[
			\frac{1}{| \wt X_3|^4 } \left( \left( \left|  \wt X_1 + \igot  \wt X_2 \right|^2 -1     \right)^2 +  \left(\left| \wt X_1^2 + \wt X_2^2 + \wt X_3^2\right|^2 -  \left| \wt X_1 - \igot \wt  X_2\right|^2     \right)^2  \right)
			\]
			is almost proper.
		\end{enumerate}
		Arguing as in the proof of Theorem \ref{th:bds}, the map 
		$
		\psi :=\pi_S\circ \Tcal \circ \wt X\colon M\setminus C\to\s_1^3
		$
	 is well defined by \ref{X3z1}, it is a $\CMC$ face by  \eqref{eq:condps}, \eqref{eq:defg}, \eqref{eq:g}, and \ref{int1},  it is almost proper by \ref{eq:properS}, and it is weakly complete by \eqref{eq:norm} and \ref{eq:apS}, see \cite{Yu97}.
	\end{proof}

\subsection*{Acknowledgments}
Castro-Infantes is partially supported by the grant PID2021-124157NB-I00 funded by MCIN/AEI/10.13039/501100011033/ ‘ERDF A way of making Europe’, Spain; and by Comunidad Aut\'{o}noma de la Regi\'{o}n de Murcia, Spain, within the framework of the Regional Programme in Promotion of the Scientific and Technical Research (Action Plan 2022), by Fundaci\'{o}n S\'{e}neca, Regional Agency of Science and Technology, REF, 21899/PI/22.

Hidalgo is partially supported by the State Research Agency (AEI) via the grant no.\ PID2020-117868GB-I00, funded by MCIN/AEI/10.13039/501100011033/, Spain.




\vspace*{-1mm}
\noindent Ildefonso Castro-Infantes

\noindent Departamento de Matem\'aticas, Universidad de Murcia, C. Campus Universitario, 9, 30100 Murcia, Spain.

\noindent  e-mail: {\tt ildefonso.castro@um.es }

\bigskip

\noindent  Jorge Hidalgo

\noindent Departamento de Geometr\'{\i}a y Topolog\'{\i}a e Instituto de Matem\'aticas (IMAG), Universidad de Granada, Campus de Fuentenueva s/n, E--18071 Granada, Spain.

\noindent  e-mail:   {\tt jorgehcal@ugr.es}

\end{document}